\documentclass[12pt]{amsart}
\usepackage{amssymb}
\usepackage{amsfonts}
\usepackage{amsbsy}
\usepackage{latexsym}
\newtheorem{theorem}{Theorem}
\newtheorem{lemma}[theorem]{Lemma}
\newtheorem{corollary}[theorem]{Corollary}

\theoremstyle{definition}
\newtheorem{definition}[theorem]{Definition}

\begin{document}
\title {Additivity of Jordan Elementary Maps on Rings}
\author{Wu Jing}
\address{Department of Mathematics and Computer Science, Fayetteville State University,  Fayetteville, NC 28301}
 \email{wjing@uncfsu.edu}

\date{May 22, 2007}
\begin{abstract}
 We prove that Jordan elementary surjective maps on rings are automatically additive. \end{abstract}
\maketitle

Elementary operators  were originally introduced by
Bre$\check{\textrm {s}}$ar and $\check{\textrm{S}}$erml (\cite
{br1115}). In the last decade, elementary maps on operator algebras
as well as on rings attracted more and more attentions. It is very
interesting that  elementary maps and Jordan elementary maps on
some algebras and rings are automatically additive. The aim of this
note is to continue to study the additivity of Jordan elementary
maps on rings and standard operator algebras. We first define Jordan
elementary maps as follows.
\begin{definition}
Let $\mathcal {R}$ and $\mathcal {R}^{\prime}$ be two rings, and let
$M\colon \mathcal {R}\rightarrow \mathcal {R}^{\prime}$ and
$M^*\colon \mathcal {R}^{\prime}\rightarrow \mathcal {R}$ be two
maps.  Call the ordered pair $(M, M^*)$ a \textit {Jordan elementary
map} of $\mathcal {R}\times \mathcal {R}^{\prime}$ if
\begin{displaymath}
 \left\{ \begin{array}{ll}
 M(aM^*(x)+M^*(x)a)=M(a)x+xM(a),\\
 M^*(M(a)x+xM(a))=aM^*(x)+M^*(x)a
\end{array}\right.
\end{displaymath}
 for all $a\in \mathcal {R}$, $x\in \mathcal {R}^{\prime}$.
\end{definition}
Note that the Jordan elementary maps defined above are different
from those in \cite{li237}.

We now  introduce some definitions and results.  Let $\mathcal {R}$
be a ring, if $a\mathcal {R}b=\{ 0\} $ implies either $a=0$ or
$b=0$, then $\mathcal {R}$ is called a \textit{prime} ring. A ring $\mathcal {R}$ is said to be \textit {$2$-torsion free} if $2a=0$ implies $a=0$.

Suppose that $\mathcal {R}$ is a ring containing  a nontrivial
idempotent $e_1$. Let $e_2=1-e_1$ (Note that $\mathcal {R}$ need not
have an identity element).  We set $\mathcal {R}_{ij}=e_i\mathcal
{R}e_j$, for $i, j=1, 2$. Then we may write $\mathcal {R}=\mathcal
{R} _{11}\oplus \mathcal {R} _{12}\oplus \mathcal {R}_{21}\oplus
\mathcal {R} _{22}$. It should be mentioned here that this
significant idea is due to Martinadale (\cite {ma695}) which has become a 
key tool in dealing with the additivity of a large number of maps on some rings and operator 
algebras. In what follows, $a_{ij}$ will denote that  $a_{ij}\in
\mathcal {R}_{ij}$ ($1\leq i, j\leq 2$).

We denote by $B(X)$ the algebra of all linear bounded operators on a
Banach space $X$. A subalgebra of
$B(X)$ is called a \textit {standard operator algebra} if it
contains all finite rank operators in $B(X)$.

Now we are ready to  state our main result of this note.
\begin{theorem}\label{theorem}
Let $\mathcal {R}$ and $\mathcal {R}^{\prime }$ be two rings.
Suppose that $\mathcal {R}$ is a $2$-torsion free ring containing a
nontrivial idempotent $e_1$ and satisfies

(i) $e_iae_j\mathcal {R}e_k=\{ 0\} $, or $e_k\mathcal {R}e_iae_j=\{
0\} $ implies $e_iae_j=0$ ($1\leq i, j, k\leq 2$), where
$e_2=1-e_1$;

(ii) If $e_2ae_2be_2+e_2be_2ae_2=0$ for each $b\in \mathcal {R}$,
then $e_2ae_2=0$.

Suppose that $M\colon {\mathcal R}\rightarrow {\mathcal R}^{\prime
}$ and $M^*\colon {\mathcal R}^{\prime }\rightarrow {\mathcal R}$
are surjective maps such that
\begin{displaymath}
 \left\{ \begin{array}{ll}
 M(aM^*(x)+M^*(x)a)=M(a)x+xM(a),\\
 M^*(M(a)x+xM(a))=aM^*(x)+M^*(x)a
\end{array}\right.
\end{displaymath}
 for all $a\in \mathcal {R}$, $x\in {\mathcal R}^{\prime }$. Then both $M$ and $M^*$ are additive. \end{theorem}

The proof of this theorem is organized as a series of lemmas. We
begin with

\begin{lemma} $M(0)=0$ and $M^*(0)=0$.
\end{lemma}
\begin{proof} We have
$M(0)=M(0M^*(0)+M^*(0)0)=M(0)0+0M^*(0)=0$.

Similarly, $M^*(0)=M^*(M(0)0+0M(0))=0M^*(0)+M^*(0)0=0$.
\end{proof}

The following lemma is very useful though the proof is simple.
\begin{lemma}\label{lu}
Let $a=a_{11}+a_{12}+a_{21}+a_{22}\in \mathcal {R}$.

(i) If $a_{ij}t_{jk}=0$ for each $t_{jk}\in \mathcal {R}_{jk} $
($1\leq i, j, k\leq 2$), then $a_{ij}=0$.

Dually, if $t_{ki}a_{ij}=0$ for each  $t_{ki}\in \mathcal {R}_{ki} $
($1\leq i, j, k\leq 2$), then $a_{ij}=0$.

(ii) If $t_{ij}a+at_{ij}\in \mathcal {R}_{ij} $ for every $t_{ij}\in
\mathcal {R}_{ij}$ ($1\leq i\not = j\leq 2$), then $a_{ji}=0$

 (iii) If $a_{ii}t_{ii}+t_{ii}a_{ii}=0$ for every $t_{ii}\in \mathcal {R}_{ii}$ ($i=1, 2$), then $a_{ii}=0$;

(iv) If $t_{jj}a+at_{jj}\in \mathcal{R}_{ij}$ for every $t_{jj}\in
\mathcal {R}_{jj}$ ($1\leq i\not =j\leq j$), then $a_{ji}=0$ and
$a_{jj}=0$.

  Dually, if  $t_{jj}a+at_{jj}\in \mathcal{R}_{ji}$ for every $t_{jj}\in \mathcal {R}_{jj}$ ($1\leq i\not =j\leq j$), then $a_{ij}=0$ and $a_{jj}=0$.
\end{lemma}
\begin{proof}
(i) It follows from  condition (i) of Theorem \ref{theorem}
directly.

(ii) Since $t_{ij}a+at_{ij}\in \mathcal {R}_{ij}$, we have
$(t_{ij}a+at_{ij})e_i=0$. Thus, $t_{ij}ae_i=0$, i.e.,
$t_{ij}a_{ji}=0$. By (i), we have $a_{ji}=0$.

(iii) For the case of $i=1$, we have $0=a_{11}e_1+e_1a_{11}=a_{11}+a_{11}=2a_{11}$, and so
$a_{11}=0$ since $\mathcal {R}$ is $2$-torsion free.

The case of $i=2$ is the same as  condition (ii) of Theorem \ref
{theorem} as $\mathcal {R}$ is $2$-torsion free.

(iv) From $t_{jj}a+at_{jj}\in \mathcal{R}_{ij}$, we have
$(t_{jj}a+at_{jj})e_i=0$. Then $t_{jj}a_{ji}=0$, and so $a_{ji}=0$.

Again, from $t_{jj}a+at_{jj}\in \mathcal{R}_{ij}$, we have $e_j(t_{jj}a+at_{jj})e_j=0$, i.e., $t_{jj}a_{jj}+a_{jj}t_{jj}=0$. By (iii), we have $a_{jj}=0$.
\end{proof}
\begin{lemma}
$M$ and $M^*$ are injective.
\end{lemma}
\begin{proof} First, we show that $M$ is injective. Let $a=a_{11}+a_{12}+a_{21}+a_{22}$ and $b=b_{11}+b_{12}+b_{21}+b_{22}$ be two elements of $\mathcal {R}$. Suppose that $M(a)=M(b)$.

For every $t_{ij}\in \mathcal {R}_{ij}$, by surjectivity of $M^*$
there exists $x(i, j)\in \mathcal {R}^{\prime}$ such that $M^*(x(i,
j))=t_{ij}$. We compute
\begin{eqnarray*}
 t_{ij}a+at_{ij}&=&M^*(x(i, j))a+aM^*(x(i, j))=M^*(x(i, j)M(a)+M(a)x(i, j))\\
&=&M^*(x(i, j)M(b)+M(b)x(i, j))=M^*(x(i, j))b+bM^*(x(i, j))\\
&=&t_{ij}b+bt_{ij}.
\end{eqnarray*}
Therefore, we have
\begin{equation}\label{00}
t_{ij}a+at_{ij}=t_{ij}b+bt_{ij}
\end{equation}

Letting $i=j=2$ in the above equality, we have
$$t_{22}a_{21}+t_{22}a_{22}+a_{12}t_{22}+a_{22}t_{22}=t_{22}b_{21}+t_{22}b_{22}+b_{12}t_{22}+b_{22}t_{22}.$$
 This implies that $t_{22}a_{21}=t_{22}b_{21}$, $a_{12}t_{22}=b_{12}t_{22}$, and $t_{22}a_{22}+a_{22}t_{22}=t_{22}b_{22}+b_{22}t_{22}$. By Lemma \ref{lu}, we get $a_{21}=b_{21}$, $a_{12}=b_{12}$ and $a_{22}=b_{22}$.

If $i=1$ and $j=2$, then equality (\ref{00}) becomes
$$t_{12}a_{21}+t_{12}a_{22}+a_{11}t_{12}+a_{21}t_{12}=t_{12}b_{21}+t_{12}b_{22}+b_{11}t_{12}+b_{21}t_{12},$$
and so $a_{11}t_{12}=b_{11}t_{12}$.Thus $a_{11}=b_{11}$.
  Therefore we can infer that $M$ is injective.

We now show that $M^*$ is injective. Let $x, y\in \mathcal
{R}^{\prime}$ such that $M^*(x)=M^*(y)$. Since $M$ is a bijection,
we may pick $a, b\in \mathcal {R}$ such that $a=M^{-1}(x)$ and
$b=M^{-1}(y)$. We write $a=a_{11}+a_{12}+a_{21}+a_{22}$ and
$b=b_{11}+b_{12}+b_{21}+b_{22}$.

For each $t_{ij}\in \mathcal {R}_{ij}$, by the surjectivity of
$M^*M$, there is a $c(i, j)\in \mathcal {R}$ such that $M^*M(c(i,
j))=t_{ij}$.

We consider
\begin{eqnarray*}
& &t_{ij}a+at_{ij}\\
&=&t_{ij}M^{-1}(x)+M^{-1}(x)t_{ij}\\
&=&M^*M(c(i, j))M^{-1}(x)+M^{-1}(x)M^*M(c(i, j))\\
&=&M^*(M(c(i, j))MM^{-1}(x)+MM^{-1}(x)M(c(i, j)))\\
&=&M^*(M(c(i, j))x+xM(c(i, j)))=c(i, j)M^*(x)+M^*(x)c(i, j)\\
&=&c(i, j)M^*(y)+M^*(y)c(i, j)=M^*(M(c(i,j))y+yM(c(i, j)))\\
&=&M^*(M(c(i, j))MM^{-1}(y)+MM^{-1}(y)M(c(i, j)))\\
&=&M^*M(c(i, j))M^{-1}(y)+M^{-1}(y)M^*M(c(i,j))\\
&=&t_{ij}M^{-1}(y)+M^{-1}(y)t_{ij}\\
&=&t_{ij}b+bt_{ij},
\end{eqnarray*}
i.e., $t_{ij}a+at_{ij}=t_{ij}b+bt_{ij}$.

With the same argument above we can get  $a_{11}=b_{11},
a_{12}=b_{12}$,  $a_{21}=b_{21}$, and $a_{22}=b_{22}$. Hence $a=b$,
equivalently, $x=y$, which completes the proof.
\end{proof}
From the above lemma we see that  both $M$ and $M^{*^{-1}}$ are
bijective.
\begin{lemma}\label{inverse}
The pair $(M^{*^{-1}}, M^{-1})$ is a Jordan elementary map on
$\mathcal {R}\times \mathcal {R}^{\prime}$. That is,
\begin{displaymath}
 \left\{ \begin{array}{ll}
 M^{*^{-1}}(aM^{-1}(x)+M^{-1}(x)a)=M^{*^{-1}}(a)x+xM^{*^{-1}}(a),\\
 M^{-1}(M^{*^{-1}}(a)x+xM^{*^{-1}}(a))=aM^{-1}(x)+M^{-1}(x)a
\end{array}\right.
\end{displaymath}
for all $a\in \mathcal {R}$, $x\in \mathcal {R}^{\prime }$.
\end{lemma}
\begin{proof}
We consider
\begin{eqnarray*}
M^*(M^{*^{-1}}(a)x+xM^{*^{-1}}(a))&=&M^*(M^{*^{-1}}(a)MM^{-1}(x)+MM^{-1}(x)M^{*^{-1}}(a))\\
&=&aM^{-1}(x)+M^{-1}(x)a,
\end{eqnarray*}
which leads to the first equality. The second one goes similarly.
\end{proof}

The following result will be used frequently in this note.
\begin{lemma}\label{add}
Let $a, b, c\in \mathcal {R}$ such that $M(c)=M(a)+M(b)$. Then
$$M^{*^{-1}}(tc+ct)=M^{*^{-1}}(ta+at)+M^{*^{-1}}(tb+bt)$$
for all $t\in \mathcal {R}$
\end{lemma}
\begin{proof} For every $t\in \mathcal {R}$, applying Lemma \ref{inverse}, we have
\begin{eqnarray*}
M^{*^{-1}}(tc+ct)&=&M^{*^{-1}}(tM^{-1}M(c)+M^{-1}M(c)t)=M^{*^{-1}}(t)M(c)+M(c)M^{*^{-1}}(t)\\
&=&M^{*^{-1}}(t)(M(a)+M(b))+(M(a)+M(b))M^{*^{-1}}(t)\\
&=&(M^{*^{-1}}(t)M(a)+M(a)M^{*^{-1}}(t))+(M^{*^{-1}}(t)M(b)+M(b)M^{*^{-1}}(t))\\
&=&M^{*^{-1}}(ta+at)+M^{*^{-1}}(tb+bt).
\end{eqnarray*}
\end{proof}

\begin{lemma}\label{lemmaiiij}
Let $a_{ii}\in \mathcal {R}_{ii}$ and $b_{ij}\in \mathcal {R}_{ij}$,
$1\leq i\not =j\leq 2$, then

(i) $M(a_{ii}+b_{ij})=M(a_{ii})+M(b_{ij})$;

(ii)
$M^{*^{-1}}(a_{ii}+b_{ij})=M^{*^{-1}}(a_{ii})+M^{*^{-1}}(b_{ij})$.
\end{lemma}
\begin{proof}
Suppose that $M(c)=M(a_{ii})+M(b_{ij})$ for some $c\in \mathcal
{R}$. For arbitrary $t_{ij}\in \mathcal {R}_{ij}$, by Lemma
\ref{add}, we have
$$M^{*^{-1}}(t_{ij}c+ct_{ij})=M^{*^{-1}}(t_{ij}a_{ii}+a_{ii}t_{ij})+M^{*^{-1}}(t_{ij}b_{ij}+b_{ij}t_{ij})=M^{*^{-1}}(a_{ii}t_{ij}).$$
 It follows that $t_{ij}c+ct_{ij}=a_{ii}t_{ij}$. By Lemma \ref{lu}, we have $c_{ji}=0$.

 Note that $t_{ij}c+ct_{ij}=t_{ij}c_{ji}+t_{ij}c_{jj}+c_{ji}t_{ij}+c_{ii}t_{ij}=t_{ij}c_{jj}+c_{ii}t_{ij}$. Therefore we have
 \begin{equation}\label{aa}
 t_{ij}c_{jj}+c_{ii}t_{ij}=a_{ii}t_{ij}.
 \end{equation}

 Now for any $t_{jj}\in \mathcal {R}_{jj}$, using Lemma \ref{add}, we have
 $$M^{*^{-1}}(t_{jj}c+ct_{jj})=M^{*^{-1}}(t_{jj}a_{ii}+a_{ii}t_{jj})+M^{*^{-1}}(t_{jj}b_{ij}+b_{ij}t_{jj})=M^{*^{-1}}(b_{ij}t_{jj}),$$
 which yields that $t_{jj}c+ct_{jj}=b_{ij}t_{jj}$. It follows from  Lemma \ref{lu} that $c_{jj}=0$. Moreover, equation (\ref{aa}) turns to be $c_{ii}t_{ij}=a_{ii}t_{ij}$, and so $c_{ii}=a_{ii}$.

 Notice that $b_{ij}t_{jj}=t_{jj}c+ct_{jj}=t_{jj}c_{ji}+t_{jj}c_{jj}+c_{ij}t_{jj}+c_{jj}t_{jj}=c_{ij}t_{jj}$. Using  Lemma \ref{lu} we see that $c_{ij}=b_{ij}$ ($i\not =j$). Therefore $c=c_{ii}+c_{ij}+c_{ji}+c_{jj}=a_{ii}+b_{ij}$. Hence
 $M(a_{ii}+b_{ij})=M(a_{ii})+M(b_{ij})$.

 By Lemma \ref{inverse} we can infer that (ii) holds.
\end{proof}

Similarly, we can get the following result.
\begin{lemma}\label{lemmaiiji}
Let $a_{ii}\in \mathcal {R}_{ii}$ and $b_{ji}\in \mathcal {R}_{ji}$,
$1\leq i\not =j\leq 2$, then

(i) $M(a_{ii}+b_{ji})=M(a_{ii})+M(b_{ji})$;

(ii)
$M^{*^{-1}}(a_{ii}+b_{ji})=M^{*^{-1}}(a_{ii})+M^{*^{-1}}(b_{ji})$.
\end{lemma}

\begin{lemma} \label{lemma121222}(i) $M(a_{12}+b_{12}c_{22})=M(a_{12})+M(b_{12}c_{22})$;

(ii)
$M^{*^{-1}}(a_{12}+b_{12}c_{22})=M^{*^{-1}}(a_{12})+M^{*^{-1}}(b_{12}c_{22})$;

(iii)  $M(a_{21}+b_{22}c_{21})=M(a_{21})+M(b_{22}c_{21})$;

(iv)
$M^{*^{-1}}(a_{21}+b_{22}c_{21})=M^{*^{-1}}(a_{21})+M^{*^{-1}}(b_{22}c_{21})$.
\end{lemma}

\begin{proof}
  Note that $a_{12}+b_{12}c_{22}=(e_1+b_{12})(a_{12}+c_{22})+(a_{12}+c_{22})(e_1+b_{12})$. We now compute
\begin{eqnarray*}
& &M(a_{12}+b_{12}c_{22})\\
&=&M((e_1+b_{12})(a_{12}+c_{22})+(a_{12}+c_{22})(e_1+b_{12}))\\
&=&M((e_1+b_{12})M^*M^{*^{-1}}(a_{12}+c_{22})+M^*M^{*^{-1}}(a_{12}+c_{22})(e_1+b_{12}))\\
&=&M(e_1+b_{12})M^{*^{-1}}(a_{12}+c_{22})+M^{*^{-1}}(a_{12}+c_{22})M(e_1+b_{12})\\
&=&M(e_1+b_{12})M^{*^{-1}}(a_{12})+M(e_1+b_{12})M^{*^{-1}}(c_{22})\\
& &+M^{*^{-1}}(a_{12})M(e_1+b_{12})+M^{*^{-1}}(c_{22})M(e_1+b_{12})\\
&=&M((e_1+b_{12})a_{12}+a_{12}(e_1+b_{12}))+M((e_1+b_{12})c_{22}+c_{22}(e_1+b_{12}))\\
&=&M(a_{12})+M(b_{12}c_{22}).
\end{eqnarray*}
  Similarly, we can get  $M(a_{21}+b_{22}c_{21})=M(a_{21})+M(b_{22}c_{21})$ from the fact that $a_{21}+b_{22}c_{21}=(e_1+c_{21})(a_{21}+b_{22})+(a_{21}+b_{22})(e_1+c_{21})$.

(ii) and (iv) follow from (i) and (iii) respectively by Lemma
\ref{inverse}.
 \end{proof}

 \begin{lemma} \label{lemma12}The following are true.

 (i) $M(a_{12}+b_{12})=M(a_{12})+M(b_{12})$;

(ii)
$M^{*^{-1}}(a_{12}+b_{12})=M^{*^{-1}}(a_{12})+M^{*^{-1}}(b_{12})$.
 \end{lemma}
 \begin{proof}
We only show (i). Suppose that $c=c_{11}+c_{12}+c_{21}+c_{22}\in
\mathcal {R}$ satisfies $M(c)=M(a_{12})+M(b_{12})$. For any
$t_{22}\in \mathcal {R}_{22}$, we have
\begin{eqnarray*}
M^{*^{-1}}(t_{22}c+ct_{22})&=&M^{*^{-1}}(t_{22}a_{12}+a_{12}t_{22})+M^{*^{-1}}(t_{22}b_{12}+b_{12}t_{22})\\
&=&M^{*^{-1}}(a_{12}t_{22})+M^{*^{-1}}(b_{12}t_{22})\\
&=&M^{*^{-1}}(a_{12}t_{22}+b_{12}t_{22}).
\end{eqnarray*}

 Note that   we apply Lemma \ref{add} in the first equality and Lemma \ref{lemma121222} in the last equality.

 Therefore we have $t_{22}c+ct_{22}=a_{12}t_{22}+b_{12}t_{22}$. Consequently,
  \begin{equation}\label{bb}
 t_{22}c_{21}+t_{22}c_{22}+c_{12}t_{22}+c_{22}t_{22}=a_{12}t_{22}+b_{12}t_{22}
 \end{equation}
  It follows that $t_{22}c_{21}=0$, and so, by Lemma \ref {lu}, $c_{21}=0$.

  Equation (\ref{bb}) also implies that $t_{22}c_{22}+c_{22}t_{22}=0$, which yields $c_{22}=0$.

  It follows from equation (\ref{bb}) that $c_{12}t_{22}=a_{12}t_{22}+t_{12}t_{22}$, and so $c_{12}=a_{12}+b_{12}$.

  To complete the proof it remains to show that $c_{11}=0$. For arbitrary $t_{12}\in \mathcal {R}_{12}$, by Lemma \ref{add}, we  compute $$M^{*^{-1}}(t_{12}c+ct_{12})=M^{*^{-1}}(t_{12}a_{12}+a_{12}t_{12})+M^{*^{-1}}(t_{12}b_{12}+b_{12}t_{12})=0.$$ Then $t_{12}c+ct_{12}=0$, consequently, $0=t_{12}c+ct_{12}=t_{12}c_{21}+t_{12}c_{22}+c_{11}t_{12}+c_{21}t_{12}=c_{11}t_{12}=0$. And so $c_{11}=0$. Therefore, $c=c_{12}=a_{12}+b_{12}$.
  \end{proof}

  \begin{lemma}\label{lemma21}
  The following hold.

  (i)   $M(a_{21}+b_{21})=M(a_{21})+M(b_{21})$;

 (ii)  $M^{*^{-1}}(a_{21}+b_{21})=M^{*^{-1}}(a_{21})+M^{*^{-1}}(b_{21})$.
 \end{lemma}
 \begin{proof}Suppose that $M(a_{21})+M(b_{21})=M(c)$ for some $c=c_{11}+c_{12}+c_{21}+c_{22}\in \mathcal {R}$.

 For any $t_{22}\in \mathcal {R}_{22}$, using Lemma \ref{add} and Lemma \ref{lemma121222}, we have
 \begin{eqnarray*}
 M^{*^{-1}}(t_{22}c+ct_{22})&=&M^{*^{-1}}(t_{22}a_{21}+a_{21}t_{22})+M^{*^{-1}}(t_{22}b_{21}+b_{21}t_{22})\\
 &=&M^{*^{-1}}(t_{22}a_{21})+M^{*^{-1}}(t_{22}b_{21})=M^{*^{-1}}(t_{22}a_{21}+t_{22}b_{21})\\
 &=&M^{*^{-1}}(t_{22}(a_{21}+b_{21}))
 \end{eqnarray*}
 which implies that $t_{22}c+ct_{22}=t_{22}(a_{21}+b_{21})$, and so \begin{equation}\label{cc}
 t_{22}c_{21}+t_{22}c_{22}+c_{12}t_{22}+c_{22}t_{22}=t_{22}(a_{21}+b_{21}).\end{equation}
 It follows that $t_{22}c_{22}+c_{22}t_{22}=0$ and $c_{12}t_{22}=0$. By Lemma \ref{lu}, we have $c_{22}=0$ and $c_{12}=0$.

 Equation (\ref{cc}) also implies that $t_{22}c_{21}=t_{22}(a_{21}+b_{21})$, and so $c_{21}=a_{21}+b_{21}$.

 We now prove that $c_{11}=0$.  To this aim, for every $t_{21}\in \mathcal {R}_{21}$, we consider $$M^{*^{-1}}(t_{21}c+ct_{21})=M^{*^{-1}}(t_{21}a_{21}+a_{21}t_{21})+M^{*^{-1}}(t_{21}b_{21}+b_{21}t_{21})=0.$$
 Thus $t_{21}c+ct_{21}=0$. Then we have $0=t_{21}c+ct_{21}=t_{21}c_{11}+t_{21}c_{12}+c_{12}t_{21}+c_{22}t_{21}=t_{21}c_{11}$, and so $c_{11}=0$. Therefore, $c=c_{21}=a_{21}+b_{21}$. The proof is complete.
   \end{proof}

   \begin{lemma} \label{lemma11}For arbitrary $a_{11}, b_{11}\in \mathcal {R}_{11}$, we have

  (i)   $M(a_{11}+b_{11})=M(a_{11})+M(b_{11})$;

 (ii)  $M^{*^{-1}}(a_{11}+b_{11})=M^{*^{-1}}(a_{11})+M^{*^{-1}}(b_{11})$.
 \end{lemma}
 \begin{proof} We only prove (i). Let $c=c_{11}+c_{12}+c_{21}+c_{22}\in \mathcal {R}$ be chosen such that $M(c)=M(a_{11})+M(b_{11})$.

 For any $t_{22}\in \mathcal {R}_{22}$, by Lemma \ref{add}, we have $$M^{*^{-1}}(t_{22}c+ct_{22})=M^{*^{-1}}(t_{22}a_{11}+a_{11}t_{22})+M^{*^{-1}}(t_{22}b_{11}+b_{11}t_{22})=0.$$ This implies that $t_{22}c+ct_{22}=0$. By Lemma \ref {lu}, we get $c_{12}=c_{21}=c_{22}=0$.

 We now show that $c_{11}=a_{11}+b_{11}$. For arbitrary $t_{12}\in \mathcal {R}_{12}$. We compute
 \begin{eqnarray*}
 & &M^{*^{-1}}(t_{12}c+ct_{12})\\
 &=&M^{*^{-1}}(t_{12}a_{11}+a_{11}t_{12})+M^{*^{-1}}(t_{12}b_{11}+b_{11}t_{12})\\
 &=&M^{*^{-1}}(a_{11}t_{12})+M^{*^{-1}}(b_{11}t_{12})=M^{*^{-1}}((a_{11}+b_{11})t_{12}).
 \end{eqnarray*}
 It follows that $t_{12}c+ct_{12}=(a_{11}+b_{11})t_{12}$. Furthermore, $c_{11}t_{12}=(a_{11}+b_{11})t_{12}$, thus $c_{11}=a_{11}+b_{11}$.

 \end{proof}
 Similarly, we have
 \begin{lemma} \label{lemma22}For arbitrary $a_{22}, b_{22}\in \mathcal {R}_{22}$, we have

  (i)   $M(a_{22}+b_{22})=M(a_{22})+M(b_{22})$;

 (ii)  $M^{*^{-1}}(a_{22}+b_{22})=M^{*^{-1}}(a_{22})+M^{*^{-1}}(b_{22})$.
 \end{lemma}

  \begin{lemma}\label{lemma1122}
 For arbitrary $a_{11}\in \mathcal {R}_{11}$ and $b_{22}\in \mathcal {R}_{22}$, the following hold.

 (i) $M(a_{11}+b_{22})=M(a_{11})+M(b_{22})$;

 (ii) $M^{*^{-1}}(a_{11}+b_{22})=M^{*^{-1}}(a_{11})+M^{*^{-1}}(b_{22})$.
 \end{lemma}
 \begin {proof}
 We only prove (i). Let  $c=c_{11}+c_{12}+c_{21}+c_{22}$ be an element of  $\mathcal {R}$ satisfying $M(c)=M(a_{11})+M(a_{22})$.

  For any $t_{22}\in \mathcal {R}_{22}$, we consider
  \begin{eqnarray*}
  M^{*^{-1}}(t_{22}c+ct_{22})&=&M^{*^{-1}}(t_{22}a_{11}+a_{11}t_{22})+M^{*^{-1}}(t_{22}b_{22}+b_{22}t_{22})\\
 &=&M^{*^{-1}}(t_{22}b_{22}+b_{22}t_{22}).
 \end{eqnarray*}
 This implies that $t_{22}c+ct_{22}=t_{22}b_{22}+b_{22}t_{22}$. Then we get  $t_{22}c_{21}=0$, $c_{12}t_{22}=0$, and $t_{22}c_{cc}+c_{22}t_{22}=t_{22}b_{22}+b_{22}t_{22}$. Again, by Lemma \ref{lu}, we have $c_{21}=c_{12}=0$,   and $c_{22}=b_{22}$.

 To complete the proof, we need to show that $c_{11}=a_{11}$. For any $t_{12}\in \mathcal {R}_{12}$, we obtain
 \begin{eqnarray*}
 M^{*^{-1}}(t_{12}c+ct_{12})&=&M^{*^{-1}}(t_{12}a_{11}+a_{11}t_{12})+M^{*^{-1}}(t_{12}b_{22}+b_{22}t_{12})\\
 &=&M^{*^{-1}}(a_{11}t_{12})+M^{*^{-1}}(t_{12}b_{22})\\
 &=&M^{*^{-1}}(a_{11}t_{12}+t_{12}b_{22}).
 \end{eqnarray*}
 Note that in the last equality we apply Lemma \ref {lemma12}.
 It follows that $t_{12}c+ct_{12}=a_{11}t_{12}+t_{12}b_{22}$, which leads to $t_{12}c_{21}+ t_{12}c_{22}+c_{11}t_{12}+c_{21}t_{12}=a_{11}t_{12}+t_{12}b_{22}$, and so $c_{11}t_{12}=a_{11}t_{12}$. Therefore, $c_{11}=a_{11}$. The proof is done.

 \end{proof}
  \begin{lemma}
 For arbitrary $a_{12}\in \mathcal {R}_{12}$ and $b_{21}\in \mathcal {R}_{21}$, we have

 (i) $M(a_{12}+b_{21})=M(a_{12})+M(b_{21})$;

 (ii) $M^{*^{-1}}(a_{12}+b_{21})=M^{*^{-1}}(a_{12})+M^{*^{-1}}(b_{21})$.
 \end{lemma}
 \begin {proof}
 Let $c=c_{11}+c_{12}+c_{21}+c_{22}\in \mathcal {R}$ be chosen such that $M(c)=M(a_{12})+M(a_{21})$.

 Now for arbitrary $t_{12}\in \mathcal {R}_{12}$, we have
 $$M^{*^{-1}}(t_{12}c+ct_{12})=M^{*^{-1}}(t_{12}a_{12}+a_{12}t_{12})+M^{*^{-1}}(t_{12}b_{21}+b_{21}t_{12})=M^{*^{-1}}(t_{12}b_{21}+b_{21}t_{12}).$$
 Therefore
 $$ t_{12}c+ct_{12}=t_{12}b_{21}+b_{21}t_{12},$$
 i.e., $t_{12}c_{11}+t_{12}c_{22}+c_{11}t_{12}+c_{21}t_{12}=t_{12}b_{21}+b_{21}t_{12}$. This implies $t_{12}c_{22}=0$, $c_{11}t_{12}=0$, and  $c_{21}t_{12}=b_{21}t_{12}$. Applying Lemma \ref {lu}, we get $c_{22}=c_{11}=0$ and $c_{21}=b_{21}$.

 We now show that $c_{12}=a_{12}$. For any $t_{21}\in \mathcal {R}_{21}$, we obtain 
 \begin{eqnarray*}
& &M^{*^{-1}}(t_{21}c+ct_{21})\\
&=&M^{*^{-1}}(t_{21}a_{12}+a_{12}t_{21})+M^{*^{-1}}(t_{21}b_{21}+b_{21}t_{21})\\
&=&M^{*^{-1}}(t_{21}a_{12}+a_{12}t_{21}).
\end{eqnarray*}
This leads to $$t_{21}c+ct_{21}=t_{21}a_{12}+a_{12}t_{21}.$$
 Multiplying the above equality from the left by $e_1$, we arrive at $c_{12}t_{21}=a_{12}t_{21}$. And so $c_{12}=a_{12}$, as desired.

 \end{proof}

  \begin{lemma} \label{lemma111221}For any $a_{11}\in \mathcal {R}_{11}$, $b_{12}\in \mathcal {R}_{12}$, and $c_{21}\in \mathcal {R}_{21}$, we have

  (i) $M(a_{11}+b_{12}+c_{21})=M(a_{11})+M(b_{12})+M(c_{21})$;

  (ii) $M^{*^{-1}}(a_{11}+b_{12}+c_{21})=M^{*^{-1}}(a_{11})+M^{*^{-1}}(b_{12})+M^{*^{-1}}(c_{21})$.

 \end{lemma}
 \begin{proof}
 We pick $d=d_{11}+d_{12}+d_{21}+d_{22}\in \mathcal {R}$ such that $M(d)=M(a_{11})+M(b_{12})+M(c_{21})$. By Lemma \ref{lemmaiiij} and Lemma \ref{lemmaiiji}, we have

\begin{equation}\label{ee}
M(d)=M(a_{11}+b_{12})+M(c_{21})
\end{equation}
and
\begin{equation}\label{ff}
M(d)=M(a_{11}+c_{21})+M(b_{12}).
\end{equation}
For any $t_{21}\in \mathcal {R}_{21}$, by Lemma \ref{add} and
equation (\ref{ee}), we have
\begin{eqnarray*}
& &M^{*^{-1}}(t_{21}d+dt_{21})\\
&=&M(^{*^{-1}}t_{21}(a_{11}+b_{12})+(a_{11}+b_{12})t_{21})+M^{*^{-1}}(t_{21}c_{21}+c_{21}t_{21})\\
&=&M^{*^{-1}}(t_{21}a_{11}+t_{21}b_{12}+b_{12}t_{21}),
\end{eqnarray*}
which yields that
\begin{equation}\label{gg}
t_{21}d+dt_{21}=t_{21}a_{11}+t_{21}b_{12}+b_{12}t_{21}.
\end{equation}
Multiplying equality (\ref{gg})  by $e_2$ from the right and the
left respectively, we get $t_{21}d_{12}=t_{21}b_{12}$ and
$t_{21}d_{11}=t_{21}a_{11}$, and so $d_{12}=b_{12}$ and
$d_{11}=a_{11}$.

We now show that $d_{21}=c_{21}$ and $d_{22}=0$. For arbitrary
$t_{12}\in \mathcal {R}_{12}$, using Lemma \ref{add} and equality
(\ref{ff}), we have
\begin{eqnarray*}
& &M^{*^{-1}}(t_{12}d+dt_{12})\\
&=&M^{*^{-1}}(t_{12}(a_{11}+c_{21})+(a_{11}+c_{21})t_{12})+M^{*^{-1}}(t_{12}b_{12}+b_{12}t_{12})\\
&=&M^{*^{-1}}(t_{12}c_{21}+a_{11}t_{12}+c_{21}t_{12}).
\end{eqnarray*}
This leads to
\begin{equation}\label{hh}
t_{12}d+dt_{12}=t_{12}c_{21}+a_{11}t_{12}+c_{21}t_{12}.
\end{equation}
Multiply equation (\ref{hh}) by $e_1$ from the right, we get
$t_{12}d_{21}=t_{12}c_{21}$ and so $d_{21}=c_{21}$.

Now equality (\ref{gg}) turns to be $d_{22}t_{21}=0$ and so
$d_{22}=0$. Therefore $d=a_{11}+b_{12}+c_{21}$.

By Lemma \ref{inverse}, we see that equality (ii) also holds.
\end{proof}

 Similarly, we have the following
 \begin{lemma}\label{lemma122122}
 For any $a_{12}\in \mathcal {R}_{12}$, $b_{21}\in \mathcal {R}_{21}$, and $c_{22}\in \mathcal {R}_{22}$, we have

  (i) $M(a_{12}+b_{21}+c_{22})=M(a_{12})+M(b_{21})+M(c_{22})$;

  (ii) $M^{*^{-1}}(a_{12}+b_{21}+c_{22})=M^{*^{-1}}(a_{12})+M^{*^{-1}}(b_{21})+M^{*^{-1}}(c_{22})$.
  \end{lemma}
 \begin{lemma}\label{lemma11122122}
  For any $a_{11}\in \mathcal {R}_{11}$, $b_{12}\in \mathcal {R}_{12}$, $c_{21}\in \mathcal {R}_{21}$, and $d_{22}\in \mathcal {R}_{22}$, the following hold.

(i)
$M(a_{11}+b_{12}+c_{21}+d_{22})=M(a_{11})+M(b_{12})+M(c_{21})+M(d_{22})$;

(ii)
$M^{*^{-1}}(a_{11}+b_{12}+c_{21}+d_{22})=M^{*^{-1}}(a_{11})+M^{*^{-1}}(b_{12})+M^{*^{-1}}(c_{21})+M^{*^{-1}}(d_{22})$.
 \end{lemma}
 \begin{proof}
 We choose $f=f_{11}+f_{12}+f_{21}+f_{22}\in \mathcal {R}$ such that $$M(f)=M(a_{11})+M(b_{12})+M(c_{21})+M(d_{22})=M(a_{11}+d_{22})+M(b_{12}+c_{21}).$$
 We compute
 \begin{eqnarray*}
 & &M^{*^{-1}}(e_1f+fe_1)\\
 &=&M^{*^{-1}}(e_1(a_{11}+d_{22})+(a_{11}+d_{22})e_1)+M^{*^{-1}}(e_1(b_{12}+c_{21})+(b_{12}+c_{21})e_1)\\
 &=&M^{*^{-1}}(a_{11}+a_{11})+M^{*^{-1}}(b_{12}+c_{21})=M^{*^{-1}}(2a_{11})+M^{*^{-1}}(b_{12})+M^{*^{-1}}(c_{21})\\
 &=&M^{*^{-1}}(2a_{11}+b_{12}+c_{21}).
 \end{eqnarray*}
 Note that in the last equality we apply Lemma \ref{lemma111221}. Then we get $e_1f+fe_1=2a_{11}+b_{12}+c_{21}$. Furthermore, we have
 $$2f_{11}+f_{12}+f_{21}=2a_{11}+b_{12}+c_{21}.$$
 Multiplying the above equality by $e_2$ from the left and the right respectively,we can infer that $f_{12}=b_{12}$, $f_{21}=c_{21}$, and $f_{11}=a_{11}$.

 We need to show $f_{22}=d_{22}$ in order to complete the proof. For any $t_{12}\in \mathcal {R}_{12}$, we consider
 \begin{eqnarray*}
 & &M^{*^{-1}}(t_{12}f+ft_{12})\\
 &=&M^{*^{-1}}(t_{12}(a_{11}+d_{22})+(a_{11}+d_{22})t_{12})+M^{*^{-1}}(t_{12}(b_{12}+c_{21})+(b_{12}+c_{21})t_{12})\\
 &=&M^{*^{-1}}(t_{12}d_{22}+a_{11}t_{12})+M^{*^{-1}}(t_{12}c_{21}+c_{21}t_{12})\\
 &=&M^{*^{-1}}(t_{12}d_{22}+a_{11}t_{12}+t_{12}c_{21}+c_{21}t_{12}).
 \end{eqnarray*}
 Consequently, $$t_{12}f+ft_{12}=t_{12}d_{22}+a_{11}t_{12}+t_{12}c_{21}+c_{21}t_{12},$$ this implies that $t_{12}f_{22}=t_{12}d_{22}$. Thus $f_{22}=d_{22}$.
 \end{proof}
  \noindent \textbf{Proof of Theorem \ref{theorem}} We first show that $M$ is additive. Let $a=a_{11}+a_{12}+a_{21}+a_{22}$ and $b=b_{11}+b_{12}+b_{21}+b_{22}$ be two arbitrary elements of $\mathcal {R}$. Then
 \begin{eqnarray*}
 & &M(a+b)\\
 &=&M((a_{11}+b_{11})+(a_{12}+b_{12})+(a_{21}+b_{21})+(a_{22}+b_{22}))\\
 &=&M(a_{11}+b_{11})+M(a_{12}+b_{12})+M(a_{21}+b_{21})+M(a_{22}+b_{22})\\
 &=&M(a_{11})+M(b_{11})+M(a_{12})+M(b_{12})+M(a_{21})+M(b_{21})+M(a_{22})+M(b_{22})\\
 &=&M(a_{11}+a_{12}+a_{21}+a_{22})+M(b_{11}+b_{12}+b_{21}+b_{22})\\
 &=&M(a)+M(b).
 \end{eqnarray*}
 That is, $M$ is additive.

We now  prove the additivity of $M^*$. For any $x, y\in \mathcal {R}^{\prime}$, there exist $c=c_{11}+c_{12}+c_{21}+c_{22}$ and $d=d_{11}+d_{12}+d_{21}+d_{22}$ in $\mathcal {R}$ such that $c=M^*(x)+M^*(y)$ and $d=M^*(x+y)$.

 For arbitrary $t_{ij}\in \mathcal {R}_{ij}$ ($1\leq i, j\leq 2$), using the additivity of $M$, we compute
 \begin{eqnarray*}
 & &M(t_{ij}c+ct_{ij})\\
 &=&M(t_{ij}(M^*(x)+M^*(y))+(M^*(x)+M^*(y))t_{ij})\\
 &=&M(t_{ij}M^*(x))+M(t_{ij}M^*(y))+M(M^*(x)t_{ij})+M(M^*(y)t_{ij})\\
 &=&M(t_{ij}M^*(x)+M^*(x)t_{ij})+M(t_{ij}M^*(y)+M^*(y)t_{ij})\\
 &=&M(t_{ij})x+xM(t_{ij})+M(t_{ij})y+yM(t_{ij})\\
 &=&M(t_{ij})(x+y)+(x+y)M(t_{ij})\\
 &=&M(t_{ij}M^*(x+y)+M^*(x+y)t_{ij})\\
 &=&M(t_{ij}d+dt_{ij}).
 \end{eqnarray*}
Therefore,
\begin{equation}\label{ii}
t_{ij}c+ct_{ij}=t_{ij}d+dt_{ij}.
\end{equation}

Letting $i=j=1$ in equality (\ref{ii}), we get
\begin{equation}\label{jj}
t_{11}c_{11}+t_{11}c_{12}+c_{11}t_{11}+c_{21}t_{11}=t_{11}d_{11}+t_{11}d_{12}+d_{11}t_{11}+d_{21}t_{11}.
\end{equation}
Multiply the above equation by $e_1$ from both sides, we get
$t_{11}c_{11}+c_{11}t_{11}=t_{11}d_{11}+d_{11}t_{11}$, which implies
that $c_{11}=d_{11}$.

Now equality (\ref {jj}) becomes
$$t_{11}c_{12}+c_{21}t_{11}=t_{11}d_{12}+d_{21}t_{11}.$$
This yields that $t_{11}c_{12}=t_{11}d_{12}$ and $c_{21}t_{11}=d_{21}t_{11}$. Then, by Lemma \ref{lu}, it follows that 
$c_{12}=d_{12}$ and $c_{21}=d_{21}$.

To show $c_{22}=d_{22}$, we let $i=j=2$ in equality (\ref{ii}) and
get
$$t_{22}c_{21}+t_{22}c_{22}+c_{12}t_{22}+c_{22}t_{22}=t_{22}d_{21}+t_{22}d_{22}+d_{12}t_{22}+d_{22}t_{22},$$
which leads to
$t_{22}c_{22}+c_{22}t_{22}=t_{22}d_{22}+d_{22}t_{22}$, and so
$c_{22}=d_{22}$. Consequently, $c=d$, hence
$M^*(x+y)=M^*(x)+M^*(y)$, which completes the proof.

For the case of Jordan elementary maps on prime rings we have the
following result.
\begin{corollary}
Let $\mathcal {R}$ be a $2$-torsion free prime ring containing a
nontrivial idempotent $e_1$, and $\mathcal {R}^{\prime }$ be an
arbitrary ring.  Suppose that $e_2ae_2be_2+e_2be_2ae_2=0$ for each
$b\in \mathcal {R}$ implies $e_2ae_2=0$. Let $M\colon {\mathcal
R}\rightarrow {\mathcal R}^{\prime }$ and $M^*\colon {\mathcal
R}^{\prime }\rightarrow {\mathcal R}$ be two surjective maps such
that
\begin{displaymath}
 \left\{ \begin{array}{ll}
 M(aM^*(x)+M^*(x)a)=M(a)x+xM(a),\\
 M^*(M(a)x+xM(a))=aM^*(x)+M^*(x)a
\end{array}\right.
\end{displaymath}
 for all $a\in \mathcal {R}$, $x\in {\mathcal R}^{\prime }$. Then both $M$ and $M^*$ are additive.
\end{corollary}
\begin{proof}
Since $\mathcal {R}$ is prime, it is easy to check that condition
(i) of Theorem \ref{theorem} holds true. Now the proof goes
directly.
\end{proof}

In particular, if a prime ring has an identity element, then we have
\begin{corollary}
Let $\mathcal {R}$ be a $2$-torsion free unital prime ring containing a   nontrivial idempotent $e_1$, and $\mathcal
{R}^{\prime }$ be an arbitrary ring.  Suppose that $M\colon
{\mathcal R}\rightarrow {\mathcal R}^{\prime }$ and $M^*\colon
{\mathcal R}^{\prime }\rightarrow {\mathcal R}$ are two surjective
maps such that
\begin{displaymath}
 \left\{ \begin{array}{ll}
 M(aM^*(x)+M^*(x)a)=M(a)x+xM(a),\\
 M^*(M(a)x+xM(a))=aM^*(x)+M^*(x)a
\end{array}\right.
\end{displaymath}
 for all $a\in \mathcal {R}$, $x\in {\mathcal R}^{\prime }$. Then both $M$ and $M^*$ are additive.
\end{corollary}

We complete this note by considering Jordan elementary maps on
standard operator algebras.

\begin{corollary}
Let $\mathcal {A}$ be a standard operator algebra on a Banach space
of dimension greater than $1$, and $\mathcal {R}$ be an arbitrary ring.
Suppose that $M\colon {\mathcal A}\rightarrow {\mathcal R}$ and
$M^*\colon {\mathcal R}\rightarrow {\mathcal A}$ are surjective maps
such that
\begin{displaymath}
 \left\{ \begin{array}{ll}
 M(aM^*(x)+M^*(x)a)=M(a)x+xM(a),\\
 M^*(M(a)x+xM(a))=aM^*(x)+M^*(x)a
\end{array}\right.
\end{displaymath}
 for all $a\in \mathcal {A}$, $x\in {\mathcal R}$. Then both $M$ and $M^*$ are additive.
\end{corollary}
\begin{proof}
Note that, by Lemma 2 in \cite {lu123}, we see that $\mathcal {A}$ satisfies conditions (i)
and (ii) of Theorem \ref{theorem}.
Now, the proof follows easily.
\end{proof}
\bibliographystyle{amsplain}

\end{document}